\documentclass[11pt,letterpaper]{amsart}
\usepackage{amsfonts, amsthm, amssymb, amsmath, stmaryrd}
\usepackage{mathrsfs,array}
\usepackage{eucal,color,times,enumerate,accents}
\usepackage[all]{xy}
\usepackage{url}
\usepackage{enumitem}
\usepackage[pdftex,bookmarksnumbered,bookmarksopen]{hyperref}

\usepackage[utf8]{inputenc}
\usepackage[T1]{fontenc}

\hypersetup{
  colorlinks,
 citecolor=black,
  linkcolor=black,
  urlcolor=black}

    \newcommand\blfootnote[1]{%
  \begingroup
  \renewcommand\thefootnote{}\footnote{#1}%
  \addtocounter{footnote}{-1}%
  \endgroup
}
  
  \newcommand{\Addresses}{{% additional braces for segregating \footnotesize
  \bigskip
  \footnotesize

\textsc{London School of Geometry and Number Theory, UCL, Department of Mathematics, Gower street, WC1E 6BT, London, UK}\par\nopagebreak
  \textit{E-mail address}, G.~Baldi: \texttt{gregorio.baldi.16@ucl.ac.uk}
  
  \textit{E-mail address}, G.~Grossi: \texttt{giada.grossi.16@ucl.ac.uk}

}}

\usepackage{bbm}

\usepackage{leftidx}

\usepackage[bottom=3.1cm, top=3.1cm, left=2.4cm, right=2.4cm]{geometry}

\input xy
\xyoption{all}

\addtocontents{toc}{\setcounter{tocdepth}{1}}

\usepackage{tikz-cd}
\usepackage{hyperref}
\hypersetup{colorlinks, linkcolor=blue}
\theoremstyle{plain}
\newtheorem{thm}{Theorem}[subsection]
\newtheorem{thmintro}{Theorem}[]
\newtheorem*{corintro}{Corollary}
\newtheorem{conj}{Conjecture}[section]
\newtheorem{lemma}[thm]{Lemma}

\newtheorem{prop}[thm]{Proposition}
\newtheorem{cor}[thm]{Corollary}
\newtheorem{cor/ex}[thm]{Corollary/Example}
\theoremstyle{definition}
\newtheorem{defi}[thm]{Definition}

\newtheorem{rmk}[thm]{Remark}
\newtheorem{question}[thm]{Question}
\theoremstyle{remark}

\numberwithin{equation}{subsection}

\makeatletter
\newenvironment{subtheorem}[1]{%
  \def\subtheoremcounter{#1}%
  \refstepcounter{#1}%
  \protected@edef\theparentnumber{\csname the#1\endcsname}%
  \setcounter{parentnumber}{\value{#1}}%
  \setcounter{#1}{0}%
  \expandafter\def\csname the#1\endcsname{\theparentnumber.\Alph{#1}}%
  \ignorespaces
}{%
  \setcounter{\subtheoremcounter}{\value{parentnumber}}%
  \ignorespacesafterend
}
\makeatother
\newcounter{parentnumber}

\newcommand{\DT}{\mathbb{S}}

\newcommand{\Tw}{\operatorname{Tw}}

\DeclareMathOperator{\Sh}{Sh}

\DeclareMathOperator{\Ext}{Ext}
\DeclareMathOperator{\End}{End}

\DeclareMathOperator{\Pic}{Pic}
\DeclareMathOperator{\Hom}{Hom}
\DeclareMathOperator{\Gal}{Gal}

\DeclareMathOperator{\Res}{Res}

\DeclareMathOperator{\im}{Im}

\DeclareMathOperator{\tr}{tr}

\DeclareMathOperator{\Gl}{GL}

\newcommand{\Gm}{\mathbb{G}_m}

\newcommand{\Z}{\mathbb{Z}}
\newcommand{\Q}{\mathbb{Q}}

\newcommand{\R}{\mathbb{R}}
\newcommand{\Pp}{\mathbb{P}}
\newcommand{\A}{\mathbb{A}}

\newcommand{\Oo}{\mathcal{O}}

\newcommand{\C}{\mathbb{C}}
\newcommand{\N}{\mathfrak{N}}

\newcommand{\Qbar}{\overline{\mathbb{Q}}}
\newcommand{\Ff}{\mathbb{F}}
\newcommand{\fr}{\text{Frob}}

\newcommand{\frw}{\text{Frob}_w}
\newcommand{\hh}{\mathcal{H}}

\newcommand\SmallMatrix[1]{{%
  \tiny\arraycolsep=0.3\arraycolsep\ensuremath{\begin{pmatrix}#1\end{pmatrix}}}}

\newcommand{\Tr}{\mathrm{Tr}}

\makeatletter
\def\@settitle{\begin{center}%
  \baselineskip14\p@\relax
  \bfseries
  \uppercasenonmath\@title
  \@title
  \ifx\@subtitle\@empty\else
     \\[1ex]\uppercasenonmath\@subtitle
     \footnotesize\mdseries\@subtitle
  \fi
  \end{center}%
}
\def\subtitle#1{\gdef\@subtitle{#1}}
\def\@subtitle{}
\makeatother

\begin{document}

\newcommand{\adjunction}[4]{\xymatrix@1{#1{\ } \ar@<-0.3ex>[r]_{ {\scriptstyle #2}} & {\ } #3 \ar@<-0.3ex>[l]_{ {\scriptstyle #4}}}}

\title{Finite descent obstruction for Hilbert modular varieties}\blfootnote{\emph{Date}. July 9, 2020}
\author{Gregorio Baldi, Giada Grossi}

\begin{abstract}
Let $S$ be a finite set of primes. We prove that a form of finite Galois descent obstruction is the only obstruction to the existence of $\mathbb{Z}_{S}$-points on integral models of Hilbert modular varieties, extending a result of D.Helm and F.Voloch about modular curves. Let $L$ be a totally real field. Under (a special case of) the absolute Hodge conjecture and a weak Serre's conjecture for mod $\ell$ representations of the absolute Galois group of $L$, we prove that the same holds also for the $\mathcal{O}_{L,S}$-points.
\end{abstract}
\maketitle

\tableofcontents

\section{Introduction}
A leading problem in arithmetic geometry is to determine whether an equation with coefficients in a number field $F$ has any solutions. Since there can be no algorithm determining whether a given Diophantine equation is soluble in the integers $\Z$, one usually tries to understand the problem under strong constraints of the geometry of the variety defined by such equation or by assuming the existence of many \emph{local} solutions. In the case of curves, for example, Skorobogatov \cite{MR1845760} asked whether the Brauer-Manin obstruction is the only obstruction to the existence of rational points. The question, or variations thereof, attracted the attention of Bruin, Harari, Helm, Poonen, Stoll and Voloch among others. In particular Helm and Voloch \cite{helmvoloch} studied a form of the finite Galois descent obstruction for the integral points of modular curves. The goal of our paper is to present a class of arbitrarily large dimensional varieties that can be treated similarly to curves. More precisely, we give sufficient conditions for the existence of $\Oo_{L,S}$-integral points on (twists of) Hilbert modular varieties associated to $K$, where both $L$ and $K$ are totally real fields. 

\subsection{What is a point of a Shimura variety?}
A point of a Shimura variety attached to a Shimura datum $(G,X)$ corresponds to a Hodge structure (once a faithful linear representation of the group $G$ is fixed). Of course not every Hodge structure can arise in this way. Even when the Shimura variety parametrises motives, there is no description of the Hodge structures coming from  geometry, nor a conjecture predicting this. However Shimura varieties have canonical models over number fields. Hence we can associate to an algebraic point a Galois representation and, conjecturally at least, we can predict which $\ell$-adic Galois representations come from geometry. This is the content of the Fontaine-Mazur conjecture \cite{MR1363495}. Our study of rational and integral points of Hilbert modular varieties begins with an attempt to understand when a suitable system of Galois representations comes from an abelian variety with $\Oo_K$-multiplication. See Section \ref{question} for a more precise formulation of the question. 

Our strategy arises from predictions of the Langlands' programme, which link the worlds of
\begin{center}
Automorphic forms $\leftrightarrow$ Motives $\leftrightarrow$ Galois representations.
\end{center}
We refer to \cite{MR3642468} for an introduction to this circle of ideas. More precisely, from a system of Galois representations that ``looks like'' the one coming from an abelian variety with $\Oo_K$-multiplication, we want to produce, via Serre's modularity conjecture, a Hilbert modular form over $L$ with Fourier coefficients in $K$. Eichler--Shimura theory attachs to this modualr form an abelian variety over $L$ with $\Oo_K$-multiplication, which will correspond to an $L$-point on the Hilbert modular variety for $K$. If $L=\Q$, Serre's conjecture is known to hold true by the work of Khare and Wintenberger \cite{MR2480604} and the Eichler--Shimura theory has been worked out by Shimura \cite{shimura}. If $L\neq \Q$, to make such a strategy work, we need to assume Serre's conjecture for the totally real field $L$ and also (a special case of) the absolute Hodge conjecture, where the latter is required by Blasius in \cite{MR2058605} in order to attach abelian varieties to Hilbert modular forms. In the next section we present in more details the main results of the paper.
\subsection{Main results}\label{mainresults}
Let $L,K$ be totally real extensions of $\Q$ and set
\begin{displaymath}
n_L:= [L:\Q], \ \ \ \ \ n_K:= [K:\Q].
\end{displaymath}
We denote by $w$ a place of $L$ and by $v$ a place of $K$. In what follows, one should think of $L$ as the \emph{field of definition} and $K$ as the \emph{Hecke field}. We denote by $\Oo_L$ and $\Oo_K$ the rings of integers of $L$ and $K$, by $L_w$ (respectively $K_v$) the completion of $L$ at $w$ (respectively of $K$ at $v$) and by $\Oo_{L_w}$ (respectively $\Oo_{K_v}$) the ring of integers of $L_w$ (respectively $K_v$). Finally $G_L$ denotes the absolute Galois group of $L$.

Let $S$ be a finite set of places of $L$ (including all archimedean places), and consider a system of Galois representations
\begin{equation}\label{conditions}
\tag{$\mathcal{S}$}
\rho_v : G_L \to \Gl_2(K_v)
\end{equation}
for every finite place $v$ of $K$, such that:
\begin{enumerate}
\item[($\mathcal{S}.1$)] $\{\rho_v\}_v$ is a weakly compatible system of Galois representations (see Definition \ref{weakcomp});
\item[($\mathcal{S}.2$)] $\det(\rho_v)=\chi_{\ell}$, where $\chi_{\ell}$ is the $\ell$-adic cyclotomic character and $v\mid \ell$;
\item[($\mathcal{S}.3$)] The residual representation $\bar{\rho}_v$ is finite flat at $w\mid\ell$, for all $v\mid \ell$ such that $\bar{\rho}_v$ is irreducible and $\ell$ is not divisible by any prime in $S$;
\item[($\mathcal{S}.4$)] $\bar{\rho}_v$ is absolutely irreducible for all but finitely many $v$;
\item[($\mathcal{S}.5$)] The field generated by the trace of $\rho_v(\fr_w)$ for every $w$ is $K$.
\end{enumerate}

To make it clear which case is conjectural and which is not, we separate the statement of our first theorem into two cases depending on whether $n_L=1$ or $n_L>1$.
\begin{subtheorem}{thmintro}\label{main}\begin{thmintro}\label{main1}
If $L=\Q$, there exists an $n_K$-dimensional abelian variety $A/\Q$ with $\Oo_K$-multiplication, such that, for every $v$, the $v$-adic Tate module of $A$, denoted by $T_v A$, is isomorphic to $\rho_v$ as representation of $G_\Q$.
\end{thmintro}
\begin{thmintro}\label{main2}
Assume $n_L>1$. Under the validity of the absolute Hodge--conjecture (more precisely Conjecture \ref{absolutehodge}) and a suitable generalisation of Serre's conjecture (Conjecture \ref{ourserreconj}), there exists an $n_K$-dimensional abelian variety $A/L$ with $\Oo_K$-multiplication, such that, for every $v$, $T_v A$ is isomorphic to $\rho_v$ as representation of $G_L$.
\end{thmintro}
\end{subtheorem}

%\begin{rmk}
%Both $K$ and $L$ have to be totally real fields because we need two Shimura varieties. In both cases we have a Hilbert modular variety. One is associated to $K$, to parametrise abelian varieties of $\Gl_2$-type with $\Oo_K$-multiplication, as in section \ref{hilb}, and the other to $L$, to have Hilbert modular forms for $L$, as in section \ref{eichlershimura}. 
%To produce an abelian variety out of a modular form, Blasius actually relate the Hilbert modular variety to a Picard modular varieties. See...???
%\end{rmk}

We apply the above to study the finite descent obstruction, as explored in \cite{MR1845760, MR2368954, MR2726726}, of Hilbert modular varieties. A recap is given in section \ref{recapobs}. More precisely, denote by $Y_K$ the Hilbert modular variety associated to $K$. Let $\mathfrak{N}$ be an ideal in $\Oo_K$ and denote by $Y_K(\mathfrak{N})$ the moduli space of $n_K$-dimensional abelian varieties, principally $\Oo_K$-polarized and with $\mathfrak{N}$-level structure (see section \ref{modulispace} for a precise definition). As a corollary of the above theorems, we prove that the finite Galois descent obstruction (as defined in section \ref{finitesection}) is the only obstruction to the existence of $S$-integral points on integral models of twists of Hilbert modular varieties, denoted by $\mathcal{Y}_K(\mathfrak{N})$, over the ring of $S$-integers of a totally real field $L$, generalising \cite[Theorem 3]{helmvoloch}. Assume that $S$ contains the places of bad reduction of $Y_K(\mathfrak{N})$. The set $\mathcal{Y}_\rho^{f-cov} (\Oo_{L,S})$ is defined in section \ref{recapobs}. We prove the following.
\begin{thmintro}\label{maincor}
If $n_L>1$, assume the conjectures of Theorem \ref{main2}. Let $\mathcal{Y}_\rho$ be the $S$-integral model of a twist of $\mathcal{Y}_K(\mathfrak{N})$, corresponding to a representation $\rho: G_L \to \Gl_2(\Oo_K/\mathfrak{N})$. If $\mathcal{Y}_\rho^{f-cov} (\Oo_{L,S})$ is non-empty then $\mathcal{Y}_\rho(\Oo_{L,S})$ is non-empty.
\end{thmintro}

In the work of Helm--Voloch, $\mathcal{Y}$ is the integral model of an affine curve. In the case of curves there are also other tools to establish (variants of) such results, without invoking Serre's conjecture. Indeed, as noticed after the proof of \cite[Theorem 3]{helmvoloch}, Stoll \cite[Corollary 8.8]{MR2368954} proved a similar result, under some extra assumptions, knowing that a factor of the Jacobian of such modular curves has finite Mordell-Weil and Tate-Shafarevich groups. The goal of this paper is to push Helm--Voloch's strategy to a particular class of varieties of arbitrarily large dimension and whose associated Albanese variety is trivial (see Theorem \ref{matsu}), thereby showing that the method could also be applied to study $L$-points.

Another reason for studying rational points of Hilbert modular varieties is the following. By \cite[Theorem 1]{MR1284798}, every smooth projective geometrically connected curve $C/\Qbar$ of genus at least two admits a non-constant $\Qbar$-morphism to either a Hilbert or a Quaternionic modular variety for $K$, where $K$ is a totally real number field depending on $C$ and the choice of a Belyi function $\beta : C\to \Pp^1$. See \cite[Remark 2]{MR1284798} for a detailed description of the ambiguities of such construction. Inspired by \cite[Theorem 5.2]{MR3544295}, where the role of the \emph{Belyi embedding} is played by the Kodaira-Parshin construction \cite{MR801927}, we have the following corollary of Theorem \ref{maincor}. For simplicity we consider rational points of projective curves, even if our main theorems are about integral points.

\begin{corintro}
Let $C/\Q$ be a smooth projective curve of genus $g \geq 2$. Assume the following:
\begin{enumerate}
\item $C(\A_\Q)^{\text{f-cov}}\neq \emptyset$;
\item There exist two totally real number fields $L,K$ and a non-constant $L$-morphism
\begin{displaymath}
f: C_L:=C\times_\Q L \to Y_K(\mathfrak{N})
\end{displaymath}
where $Y_K(\mathfrak{N})$ denotes, as above, the Hilbert modular variety for $K$ of some level $\mathfrak{N}$;
\item If $L\neq \Q$, the conjectures of Theorem \ref{main2} hold true.
\end{enumerate}
Then $C(\Q)\neq \emptyset$.
\end{corintro}
Stoll \cite[Conjecture 9.1]{MR2368954} conjectured that every smooth projective curve is \emph{very good}. That is the closure of $C(\Q)$ in the adelic points of $C$ is equal to $C(\A_\Q)^{\text{f-ab}}$. For the definition of $C(\A_\Q)^{\text{f-ab}}$ we refer to Definition 6.1. in \emph{op. cit.}. For the moment we just need to know that
\begin{displaymath}
C(\Q)\subset C(\A_\Q)^{\text{f-cov}} \subset C(\A_\Q)^{\text{f-ab}} \subset C(\A_\Q).
\end{displaymath}
Hence Stoll's conjecture predicts the following implication
\begin{displaymath}
C(\A_\Q)^{\text{f-cov}}\neq \emptyset \Rightarrow C(\Q)\neq \emptyset,
\end{displaymath}
which we are able to prove, as a consequence of Theorem \ref{maincor}, for curves satisfying (1),(2) and (3) from the above corollary.  
\begin{proof}
Fix a point $(P_w)\in C(\A_\Q)^{\text{f-cov}}$, which is not empty by (1). Let $X/L$ be the image of $C$ in $Y:=Y_K(\mathfrak{N})$ under the map $f$ of assumption (2). Notice that, since $C(\A_\Q)^{\text{f-cov}}\neq \emptyset$, also $C(\A_L)^{\text{f-cov}}\neq \emptyset$, and therefore, by \cite[Proposition 5.9.]{MR2368954}, $X(\A_L)^{\text{f-cov}}$ and $Y(\A_L)^{\text{f-cov}}$ are not empty. The untwisted version of Theorem \ref{maincor}, for a suitable choice of a finite set $S$ of places of $L$, implies that $Y(L)\neq \emptyset$ (this is the only step where (3) is needed). The proof of Theorem \ref{maincor} (cf. section \ref{proofmaincor}) actually guarantees the existence of a point in $Y(L)$ inducing the fixed $f(P_w)\in Y (\A_{L,S})^{\text{f-cov}}$. Since $Y(L)$ injects into $Y(\A_{L,S})^{\text{f-cov}}$, there exists a unique $Q \in Y(L)$ inducing $f(P_w)\in  Y(\A_L)^{\text{f-cov}}$. Moreover, by construction, $f(P_w)$ lies in $X(\A_L)^{\text{f-cov}}$, and $X(L)=Y(L)\cap X(\A_L)^{\text{f-cov}}$. Eventually we conclude that $Q$ lies in $X(L)$.

Let $Z$ be the finite $\Q$-subscheme of $C$ given by the $G_{\Q}$-orbit of the pull-back of $Q$ along the surjective map 
\begin{displaymath}
C_L \to X.
\end{displaymath}
By construction, $(P_w)\in C(\A_\Q)^{\text{f-cov}} \subset C(\A_\Q)^{\text{f-ab}} $ lies in $Z(\A_\Q)$. In particular \cite[Theorem 8.2]{MR2368954} implies that $Z(\Q)\neq \emptyset$. That is, $C$ has at least one $\Q$-rational point, concluding the proof of the corollary.
\end{proof}

\subsection{Conjectures} We briefly state the conjectures appearing in Theorem \ref{main2}.
\subsubsection{Absolute Hodge Conjecture}
We only give a brief overview for the purpose of understaning the conjecture. For more details we refer to section 6 of Deligne-Milne's paper \cite{MR654325}, where Deligne's category of \emph{absolute motives} is described. Let $X,Y /\C$ be smooth projective varieties. A morphism of Hodge structures between their Betti cohomology groups corresponds to a Hodge class in the cohomology of $X\times Y$:
\begin{displaymath}
\Hom (H^*_{Betti}(X,\Q), H^*_{Betti}(Y,\Q) ) \cong H^{2*}_{Betti}(X\times Y, \Q).
\end{displaymath}

We say that a Hodge class  $\alpha \in H^{2i}_{Betti}(X\times Y, \Q)$, or a morphism of Hodge structures between their $H^i$'s, is absolute Hodge if, for every automorphism $\sigma$ of $\C$, the class $\alpha ^\sigma \in H^{2i}(X^\sigma\times Y^\sigma, \C)$ is again a Hodge class. With this definition we can split the classical Hodge conjecture in two parts:
\begin{center}
Hodge classes = Absolute Hodge classes = Algebraic cycles.
\end{center}
For the purpose of this paper, the following is enough.
\begin{conj}\label{absolutehodge}
If $X,Y /\C$ are smooth projective complex varieties such that, for some $i$, we have an isomorphism of Hodge structures
\begin{displaymath}
H^i_{Betti}(X,\Q)\cong H^i_{Betti}(Y,\Q),
\end{displaymath}
then there exists an absolute Hodge class inducing this isomorphism.
\end{conj}
More precisely, Conjecture \ref{absolutehodge} will be applied when $X$ is a Picard modular variety and $Y$ is an abelian variety. 
\subsubsection{Serre's weak conjecture over totally real fields}
We now explain the version of Serre's conjecture we need to assume, to obtain the main theorems when $L\neq \Q$ (see for example \cite[Conjecture 1.1]{MR2730374}, where it is referred to as a folklore generalisation of Serre's conjecture). For more details we refer to the introduction of \cite{MR2730374} and references therein. Given a prime $\ell$ here we denote by $\Ff_\ell$ a finite field with $\ell$ elements and by $\overline{\Ff}_\ell$ a fixed algebraic closure of $\Ff_\ell$.
\begin{conj}\label{ourserreconj}
Let $\ell$ be any odd prime and $\overline{\rho}: G_L \to \Gl_2( \overline{\Ff}_\ell)$ be an irreducible and
totally odd\footnote{Here totally odd means that $ \det(\overline{\rho}(c)) = -1$ for all $n_L$ complex conjugations $c$. } Galois representation. Then there exists some Hilbert modular eigenform $f$  for $L$ such that $\overline{\rho}$ is isomorphic to the reduction mod $\lambda$ of $\rho_{f,\lambda}$, where $\rho_{f,\lambda}$ is the $\lambda$-adic Galois representation attached to $f$ and $\lambda$ is a prime of the Hecke field of $f$ dividing $\ell$.
\end{conj}

\begin{rmk}\label{weak}
This is usually referred to as \emph{weak} Serre's conjecture, because there is no explicit recipe to compute the weight $k(\overline{\rho})$ and the level $N(\overline{\rho})$. It has been proven \cite{geeliusavitt, fuji} that the \emph{refined} version follows from the weak version under some assumptions. We state the results we need in Theorem \ref{thmweightcond}.
\end{rmk}

When $L=\Q$, this was proven by Khare and Wintenberger \cite{MR2480604}. When $L\neq \Q$, Conjecture \ref{ourserreconj} is known when the coefficient field is $\Ff_3$ (Langlands-Tunnell \cite{MR574808, MR621884}), but we really need to assume that the conjecture holds for all (but finitely many) $\ell$'s. Our strategy follows indeed the lines of the proof of modularity theorems assuming Serre's conjecture: starting from a system of Galois representations, we produce a Hilbert modular form whose Fourier coefficients are equal to the traces of Frobenii modulo infinitely many primes and hence are equal as elements of $\Oo_K$.

Finally, a potential version of the above conjecture has been proven in \cite[Theorem 1.6]{MR1954941}. There Taylor proves a potential modularity result: if $\bar{\rho}: G_{L}\to \Gl_2(\overline{\mathbb{F}}_{\ell})$ is a totally odd irreducible representation with determinant equal to the cyclotomic character, then there exists $L'/L$ a finite totally real Galois extension such that all the primes of $L$ above $\ell$ split in $L'$ and there exists $f$ a Hilbert modular form for $L'$ such that $\bar{\rho}_{f,\lambda'}$ is isomorphic to $\bar{\rho}$ restricted to $G_{L'}$.

\subsection{Related work}
We compare our results with \cite[Theorem 3.1 and Theorem 3.7]{MR3544295} (later also extended to the moduli space of K3 surfaces by the first author \cite[Theorem 1.3]{MR3951650} and Klevdal \cite[Theorem 1.1]{klevdal}, where a finite extension of the base field is however required). In the approach of Patrikis, Voloch and Zarhin there are no restrictions on the base field, wheras here it is crucial that $L$ is a totally real field. We believe that it is easier to make the results of this paper unconditional. We notice here that the absolute Hodge conjecture is not enough for such papers. In \cite{MR3544295, MR3951650, klevdal} the Hodge conjecture is not only needed to descend complex abelian varieties over number fields. Finally the version of Serre's conjecture we are assuming here is always about $\Gl_2$-coefficients, and so is certainly easier than the full Fontaine-Mazur conjecture \cite{MR1363495}.

Thanks to the recent breakthroughs on potential modularity over CM fields \cite{potential}, it should be possible to extract from the main result of \cite[Section 7.1]{potential} the following. Let $\{\rho_v\}_v$ be a compatible system as in section \ref{mainresults}, and let $n_L>1$. Under the validity of Conjecture \ref{absolutehodge} there exist a totally real extension $L'/L$ and $n_K$-dimensional abelian variety $A/L'$ with $\Oo_K$-multiplication, such that, for every $v$, $T_v A$ is isomorphic to $\rho_v$ as representations of $G_{L'}$.

Since Conjecture \ref{absolutehodge} can be avoided in many interesting cases, as recalled in Remark \ref{rmk2}, a potential but unconditional version of Theorem \ref{maincor} can therefore be obtained. Unfortunately the extension $L'/L$ depends on the system of Galois representations $\{\rho_v\}_v$, and it is not easy to control a priori its degree over $L$.
 
\subsection{Outline of the paper}
In Section \ref{recap} we collect all the results we need about Hilbert modular forms (especially how Eichler-Shimura works in this setting). In Section \ref{sectionserre}, which is the heart of the paper, we prove Theorems \ref{main}.A-\ref{main}.B. We then explain how these results are related to the finite descent obstruction for Hilbert modular varieties in Section \ref{finitesection}, eventually proving Theorem \ref{maincor}.
\subsection{Acknowledgements} It is a pleasure to thank Toby Gee for helpful discussions about the weight part of Serre's conjecture and the Taylor--Wiles assumption. We thank Matteo Tamiozzo for his comments on an earlier draft of this paper. We would also like to express our gratitude to the anonymous referees whose comments improved the paper. This work was supported by the Engineering and Physical Sciences Research Council [EP/ L015234/1], the EPSRC Centre for Doctoral Training in Geometry and Number Theory (The London School of Geometry and Number Theory) and University College London.

\section{Recap on Hilbert modular varieties and modular forms}\label{recap}
We recall some general facts about Shimura varieties. The reader interested only in Hilbert modular varieties may skip section \ref{shimura}, which is not fundamental for the main results. We then focus on Hilbert modular varieties and Hilbert modular forms, explaining how they give rise to certain principally polarised abelian varieties.
\subsection{A question on rational points on Shimura varieties}\label{shimura}
Let $\DT$ denote the real torus $\Res_{\C / \R} (\Gm)$, and $\A_\Q^f $ be the finite adeles of $\Q$. A Shimura datum is a pair $(G,X)$ where $G$ is a reductive $\Q$-algebraic group and $X$ a $G(\R)$-orbit in the set of morphisms of $\R$-algebraic groups $\Hom(\DT, G_\R)$, satisfying the Shimura-Deligne axioms (\cite[Conditions 2.1.1(1-3)]{deligneshimura}); furthermore, in what follows we also assume that $G$ is the generic Mumford-Tate group on $X$. The Shimura-Deligne axioms imply that the connected components of $X$ are hermitian symmetric domains and that faithful representations of $G$ induce variations of polarisable $\Q$-Hodge structures on $X$. Let $\widetilde{K}$ be a compact open subgroup of $G(\A_\Q^f)$ and set
\begin{displaymath}
\Sh_{\widetilde{K}}(G,X) := G(\Q)  \backslash \left ( X \times G(\A_\Q^f) / \widetilde{K} \right).
\end{displaymath}
Let $X^+$ be a connected component of $X$ and $G(\Q)^+$ be the stabiliser of $X^+$ in $G(\Q)$. The above double coset set is a disjoint union of quotients of $X^+$ by the arithmetic groups $\Gamma_g:=G(\Q)^+ \cap gKg^{-1}$, where $g$ runs through a set of representatives for the finite double coset set $G(\Q)^+ \backslash G(\A_\Q^f )/K$. Baily and Borel \cite{MR0216035} proved that $\Sh_{\widetilde{K}}(G,X)$ has a unique structure of a quasi-projective complex algebraic variety. Thanks to the work of Borovoi, Deligne, Milne and Milne-Shih, among others, the $\C$-scheme
\begin{displaymath}
\Sh (G,X)= G(\Q)  \backslash \left ( X \times G(\A_\Q^f) \right),
\end{displaymath}
together with its $G(\A_\Q^f)$-action, can be naturally defined over a number field $E:= E(G,X) \subset \C$ called the \emph{reflex field} of $(G,X)$. That is there exists an $E$-scheme $\Sh (G,X)_E$ with an action of $G(\A_\Q^f)$ whose base change to $\C$ gives $\Sh (G,X)$ with its $G(\A_\Q^f)$-action.

Let $F$ be a finite extension of $E$ such that there exists a point $x \in \Sh_{\widetilde{K}}(G,X)_E(F)$. To such a point we can naturally associate a continuous group homomorphism
\begin{displaymath}
\rho_x : G_F:=\Gal(\overline{F}/F) \to \widetilde{K} \subset G(\A_\Q^f).
\end{displaymath}
This paper is motivated by the following.
\begin{question}\label{question}
Let $F$ be a field as above and let $\rho : G_F \to \widetilde{K} \subset G(\A_\Q^f)$ be a Galois representation. What are necessary and sufficient conditions such that there exists $x \in \Sh_{\widetilde{K}}(G,X)_E(F)$ and $\rho = \rho_x$?
\end{question}
When the Shimura variety has a natural interpretation as a moduli space of motives, the above question is naturally related to the Fontaine--Mazur conjecture \cite{MR1363495}. Indeed both aim to predict when a Galois representation comes from the $\ell$-adic (or adelic in our case) realisation of a motive. Even when they are not motivical (see \cite[Remark 1.6]{2018arXiv180204682B} for a discussion about this), representations arising in this way enjoy nice properties. For an example of geometric flavour we refer to \cite[Theorem 1.3]{2018arXiv180204682B}.

\subsection{Hilbert modular varieties}\label{hilb}
Let $F/\Q$ be a totally real extension of degree $n_F$ and fix $\{\sigma_i\}_{i=1}^{n_F}$ the set of real embeddings of $F$ into $\C$. We let $G$ be the $\Q$-algebraic group obtained as the Weil restriction of $\Gl_2$ from $F$ to $\Q$ and $X$ be $n_F$ copies of $\hh^{\pm}=\{\tau\in\C: \im(\tau)\neq 0\}$, on which $G(\Q)=\Gl_2(F)$ acts on the $i$th component via $\sigma_i$. That is
\begin{displaymath}
\left(\left( \begin{matrix} 
a & b \\
c & d 
\end{matrix}\right)\cdot (\tau_1,\dots,\tau_{n_F})\right)_i = \frac{\sigma_i(a) \tau_i+\sigma_i(b)}{\sigma_i(c)\tau_i+\sigma_i(d)}.
\end{displaymath}

In this case, the reflex field of $(G,X)$ is $\Q$, and the set of geometrically connected components of the Shimura variety $S:=\Sh_{G(\widehat{\Z})}(G,X)$ is $\Pic(\Oo_F)^+$, where $\widehat{\Z}$ denotes the profinite completion of $\Z$ (different choices of level structure will appear later).
\begin{rmk}\label{rmk1}
To obtain a Shimura variety from the above construction, it is fundamental that $F$ is totally real. Indeed if $F$ is a number field, $G$ an algebraic $F$-group, then the real points of $\Res_{F/\Q}G$ have a structure of Hermitian symmetric space if and only if $F$ is a totally real field and the symmetric space associated to each real embedding of $F$ is Hermitian. 
\end{rmk}
It is interesting to notice here a first difference between modular curves (i.e. when $F=\Q$) and higher dimensional Hilbert modular varieties. We recall the following folklore result, see  \cite[Section 2.3.2.]{MR2058605} to see how it follows from Matsushima's formula \cite[Theorem VII.5.2]{MR1721403}.
\begin{thm}\label{matsu}
Let $(G,X)$ be a Shimura datum as above and let $\widetilde{K}$ be a neat\footnote{A neat subgroup of $G(\A_{\Q}^f)$ is an open compact subgroup such that every element of $\widetilde{K}\cap G(\Q)$ is neat. Recall that an element $g$ of $G(\Q)$ is called neat if the subgroup of $\overline{\Q}^{\times}$ generated by the eigenvalues of $g$ in some faithful representation $V$ of $G$ is free (that is, there is no nontrivial elements of finite order). This is independent of $V$, as all faithful representations $W$ are obtained from $V$ via sums, tensor products, duals and subquotients, hence the group in question is the set of eigenvalues that occur in $W$.} subgroup of $G(\A_{\Q}^f)$. Consider $S_{\widetilde{K}}$, the Shimura variety associated to $(G,X)$ and $\widetilde{K}$. Unless $n_F=1$, the first group of Betti cohomology of $S_{\widetilde{K}}$ with rational coefficients is trivial. In particular there are no non-constant maps from $S_{\widetilde{K}}$ to an abelian variety.
\end{thm}

To have a better interpretation as moduli space we actually consider the subgroup $G^*$ of $G$ given by its elements whose determinant is in $\Q$. More precisely we let $G^*$ be the pull-back of
\begin{displaymath}
\det: G \to \Res _{\Oo_F /\Z}\Gm
\end{displaymath}
to $\Gm/\Q$. The Shimura variety $Y_F:=\Sh_{G^*(\widehat{\Z})}(G^*,X^*)(\C)$ is connected and comes with a finite map to $S/\C$. It is a quasi projective $n_F$-dimensional $\Q$-scheme.

In the next section we present the moduli problem solved by $Y_F$. It will be clear also from such moduli interpretation that the reflex field of $Y_F$ is the field of rational numbers. 
\subsubsection{Hilbert modular varieties as moduli spaces}\label{modulispace}
As explained for example in \cite[Section 3]{hilbertbas}, the Shimura variety $Y_F$ represents the (coarse) moduli space for triplets $(A, \alpha, \lambda)$ where:
\begin{itemize}
\item $A$ is a complex abelian variety of dimension $n_F$;
\item $\alpha : \Oo_F \hookrightarrow \End (A)$ is a morphism of rings;
\item $\lambda : A \to A^*$ is a principal $\Oo_F$-polarisation.
\end{itemize}
By $A^*$ we denoted the $\Oo_F$-dual abelian variety of $A$, i.e. it is defined as $\Ext^1(A,\Oo_F\otimes \Gm)$. Otherwise one can obtain such abelian variety considering $A^\vee$ (the dual of $A$, in the standard sense) and tensoring it over $\Oo_F$ with the different ideal of the extension $F/\Q$. By principal $\Oo_F$-polarisation we mean an isomorphism $\lambda : A \to A^*$, such that the induced map $A\to A^\vee$ is a polarisation.

Furthermore, the Shimura variety of level
\begin{displaymath}
U_0(\mathfrak{N}):= \left\lbrace \gamma\in G(\widehat{\Z}): \gamma \equiv \SmallMatrix{*&*\\0&1} \text{ mod } \mathfrak{N}\right\rbrace,
\end{displaymath}
where $\mathfrak{N}$ is an integral ideal of $\Oo_F$, parametrises triplets as above, equipped with a $\mathfrak{N}$-level structure as follows. We fix an isomorphism of $\Oo_F$-modules
\begin{displaymath}
\left(\Oo_F/\mathfrak{N} \Oo_F \right)^2 \to A[\mathfrak{N}]
\end{displaymath}
making the following diagram commutative:
\begin{displaymath}
\begin{tikzcd}
\left(\left(\Oo_F/\mathfrak{N} \Oo_F \right)^2\right)^2 \arrow{d}{\psi_{\mathfrak{N}}} \arrow{r}{} &  A[\mathfrak{N}]^2 \arrow{d}{e_{\lambda,\mathfrak{N}}}\\
\Oo_F\otimes \Z/N\Z \arrow{r}{} & \Oo_F \otimes \mu_N,
\end{tikzcd}
\end{displaymath}
where $(N)=\Z\cap \mathfrak{N}$, $\psi_{\mathfrak{N}}$ is the pairing given by $\SmallMatrix{0 &1\\-1&0}$, $e_{\lambda,\mathfrak{N}}$ is the perfect Weil pairing on $A[\mathfrak{N}]$ induced by the $\Oo_F$-polarisation $\lambda$ and $\Oo_F\otimes \Z/N\Z\to \Oo_F \otimes \mu_N$ is an arbitrarily chosen isomorphism. When a level structure is needed, we always assume that $N>3$ to have a \emph{fine} moduli space. 
A rational point of $Y_F(\mathfrak{N}):=\Sh_{U_0(\mathfrak{N})\cap G^*}(G^*,X)$ represents then a triple as above together with such level structure. In section \ref{integralmodels}, we will see a similar description for $\Oo_{L,S}$-points of an \emph{integral} model of $Y_F(\mathfrak{N})$. 
 
\subsection{Eichler-Shimura theory}\label{eichlershimura}
We discuss Eichler-Shimura theory for classical and Hilbert modular forms, reviewing results that attach opportune abelian varieties to modular forms. 
\subsubsection{Classical modular eigenform}
The following is \cite[Theorem 7.14, page 183 and Theorem 7.24, page 194]{shimura}.
\begin{thm}[Shimura]\label{eichlershimurathm}
Let $f$ be a holomorphic newform of weight $2$ with rational Fourier coefficients $(a_n(f))_n$. There exists an elliptic curve $E/\Q$ such that, for all primes $p$ at which $E$ has good reduction one has
\begin{displaymath}
a_p(f) = 1- N_p(E) + p,
\end{displaymath}
where $N_p(E)$ denotes the number of points of the reduction mod $p$ of $E$, over the field with $p$-elements. In other words, up to a finite number of Euler factors, $L(s,E/\Q)$ and $L(s,f)$ coincide. 

More generally, let $K(f)$ be the subfield of $\C$ generated over $\Q$ by $(a_n(f))_n$ for all $n$. Then there exists an abelian variety $A/\Q$ and an isomorphism $K(f) \cong\End^0(A)$ with the following properties:
\begin{itemize}
\item $\dim (A)=[K(f):\Q]$;
\item Up to a finite number of Euler factors at primes at which $A$ has good reduction $L(s,A/\Q,K(f))$ coincides with $L(s,f)$;
\end{itemize}
where the $L$-function $L(s,A/\Q,K(f))$ is defined by the product of the following local factors where $v$ is a prime of $K(f)$ not dividing $\ell$
\begin{displaymath}
\det(1-\ell^{-s}\operatorname{Frob}_{\ell}|T_v(A)).
\end{displaymath}
\end{thm}

Shimura considers the Jacobian of the modular curve of level $N$ and takes the quotient by the kernel of the homomorphism giving the Hecke action on $f$. What happens if we want to produce an abelian variety with such properties, defined over our totally real field $F$, when $\Q\subsetneq F$? Here is where Hilbert modular forms come into play. In the next section we recall what we need from such theory, and explain Blasius' generalisation of Theorem \ref{eichlershimurathm} and why the absolute Hodge conjecture is needed.

\subsubsection{Hilbert modular forms for $F$}\label{hilbertmodularforms}
Let $\hh_F$ denote $n_F$ copies of the upper half plane $\hh^+$. We consider subgroups $\Gamma \subset \Gl_2(\Oo_F)$ of the form $U_0(\mathfrak{N}) \cap G(\Q)^+$. Moreover for $\lambda\in F$ and $\underline{r}=(r_1,\dots, r_{n_F})\in \Z^{n_F}$, we write $\lambda^{\underline{r}}=\lambda_1^{r_1}\cdots\lambda_{n_F}^{r_{n_F}}$, where $\lambda_i=\sigma_i(\lambda)$. Similarly if $\underline{\tau}=(\tau_1,\dots,\tau_{n_F})\in \hh_F$, we write $\underline{\tau}^{\lambda}=\tau_1^{r_1}\cdots\tau_{n_F}^{r_{n_F}}$.
\begin{defi}
A Hilbert modular form of level $\mathfrak{N}$ and weight $(\underline{r},w)\in \Z^{n_F}\times \Z$, with $r_i\equiv w \mod 2$ (and trivial nebentype character) is a holomorphic function $f: \hh_F \rightarrow \C$ such that 
\begin{displaymath}
f(\gamma \cdot \underline{\tau})= (\det\gamma)^{-\underline{r}/2}(c\underline{\tau}+d)^{\underline{r}}f(\underline{\tau}),
\end{displaymath}
for every $\gamma=\SmallMatrix{a & b \\
c & d} \in U_0(\N) \cap G(\Q)^+$ and for every $\underline{\tau}=(\tau_1,\dots,\tau_{n_F})\in\hh_F$.
\end{defi}
%
%\begin{rmk}[Fourier expansion]\label{hilbfourier} For every $b\in \Oo_F$, since $\gamma=\SmallMatrix{1 & b\\
%0 & 1}$ is in the congruence subgroups considered above, a Hilbert modular form $f$
%%, we get
%%\begin{displaymath}
%%f(\tau_1 + b_1,\dots,\tau_{n_F} + b_{n_F})= f(\tau_1,\dots,\tau_{n_F} ).
%%\end{displaymath}
%%Hence $f$ is a holomorphic function on $\hh_F$ 
%invariant under the lattice $\Oo_F$. We therefore get a Fourier expansion over $\Oo_F^{\vee}=\mathfrak{d}^{-1}$. We should also notice that we are not forcing the ``holomorphicity at cusps'' as in the case of classical modular forms, but in higher dimension cases this is automatically true thanks to Koecher's principle. We hence have that the sum is of the form
%\begin{displaymath}
%f(\underline{\tau})=a_0+\sum_{\mu\in (\mathfrak{d}^{-1})^+} a_{\mu}\cdot e^{2\pi i\sum \mu_i\tau_i}.
%\end{displaymath}
%\end{rmk}
Since Hilbert modular forms are holomorphic functions on $\hh_F$ invariant under the lattice $\Oo_F$, they admit a $q$-expansion over $\Oo_F^{\vee}=\mathfrak{d}^{-1}$, where $q=e^{2\pi i \sum \tau_i}$, see \cite[Definition 3.1]{Goren} for more details. 

\subsubsection{Hecke operators and Hilbert eigenforms}
On the space of Hilbert modular forms of level
\begin{displaymath}
U_0(\N) \cap G(\Q)^+,
\end{displaymath}
one has Hecke operators $T(\mathfrak{n})$ for every integral ideal of $\Oo_F$ coprime with $\N$. The definition is analogous to the one for classical modular forms. For example, if $\mathfrak{p}$ does not divide $\N$ and $x$ is a totally positive generator of $\mathfrak{p}$, one defines
\begin{displaymath}
(T(\mathfrak{p})f)(\underline{\tau}):=\text{Nm}(\mathfrak{p}) f(x\cdot\underline{\tau})+ \frac{1}{\text{Nm}(\mathfrak{p})}\sum_{a\in\Oo_F/\mathfrak{p}}f(\gamma_a \cdot \underline{\tau}),
\end{displaymath} 
where $\gamma_a:= \SmallMatrix{1 & a \\
0 & x}$. We recall the following.
\begin{defi}
A cuspidal Hilbert modular form (i.e. such that the $0$-th Fourier coefficient $a_0(f)$ vanishes) is an eigenform if it is an eigenvector for every Hecke operator $T(\mathfrak{n})$. 
\end{defi}

As in the case of classical modular forms, if $f$ is an eigenform, normalised so that $a_1(f) = 1$, then the eigenvalues of the Hecke operators are the Fourier coefficients, i.e. $T(\mathfrak{n})f = a_{\mathfrak{n}}(f)\cdot f$; moreover they are algebraic integers lying in the number field $K(f):=\Q((a_{\mathfrak{n}}(f))_{\mathfrak{n}})$, as shown in \cite[$\S$2]{shimurahilb}.

\subsubsection{Eichler--Shimura for Hilbert modular forms}
Blasius and Rogawski \cite{blasiusrogawski}, Carayol \cite{MR870690} and Taylor \cite{MR1016264} proved that to any Hilbert eigenform, one may attach representations of $G_F$, similarly to the classical case. More precisely, one has the following result.
\begin{thm}
If $f$ is a Hilbert eigenform for $F$ of weight $(\underline{r},t)$, level $\mathfrak{N}$ and trivial nebentype character and $K(f)$ is the number field generated by its eigenvalues, then for every finite place $\lambda$ of $K(f)$ there is an irreducible 2-dimensional Galois representation
\begin{displaymath}
\rho_{f,\lambda}: G_F \to \Gl_2(K(f)_{\lambda})
\end{displaymath}
such that for every prime $w\nmid \mathfrak{N}\operatorname{Nm}_{K(f)/\Q}(\lambda)$ in $F$, $\rho_{f,\lambda}$ is unramified at $w$ and
\begin{displaymath}
\det(1-X\rho_{f,\lambda}(\operatorname{Frob}_{w}))=X^2-a_w(f)X+\operatorname{Nm}_{F/\Q}^{t-1}(w).
\end{displaymath}
\end{thm}

Assume that $f$ is of weight $(2,\dots ,2 )$. As in the classical case, we would like to have such Galois representations to be attached to opportune abelian varieties. The existence of abelian varieties associated to $f$ was first considered by Oda in \cite{oda} and Blasius gave a conjectural solution to such problem.
\begin{thm}[Blasius]\label{generalblasius}
Let $f$ be a Hilbert eigenform for $F$ of parallel weight 2. Denote by $K(f)$ the number field generated by the $a_w(f)$ for all $w$. Assume Conjecture \ref{absolutehodge}. There exists a $[K(f):\Q]$-dimensional abelian variety $A_f/F$ with $\Oo_{K(f)}$-multiplication such that for all but finitely many of the finite places $w$ of $F$ at which $A_f$ has good reduction, we have
\begin{displaymath}
L(s,A_f,K(f))=L(f,s),
\end{displaymath}
where the $L$-function $L(s,A_f,K(f))$ is defined by the product of local factors 
\begin{displaymath}
\det(1-\operatorname{Nm}(w)^{-s}\operatorname{Frob}_w|T_v(A)^{I_w}),
\end{displaymath} 
where $v$ is a prime of $K$, $w$ is a prime of $F$ and $w$ and $v$ lay above distinct rational primes.
\end{thm}
\begin{proof} If $K(f)=\Q$, this is precisely \cite[Theorem 1 on page 3]{MR2058605}. As noticed by Blasius \cite[1.10]{MR2058605}, the proof easily adapts to the general case (where the necessary changes are hinted in the remarks in section \cite[5.4., 5.7.,  7.6.]{MR2058605}). 
\end{proof}
\begin{rmk}
The proof is completely different from the one of Shimura, since, as noticed in Theorem \ref{matsu}, we can not obtain a non-trivial abelian variety as quotient of the Albanese variety of a Hilbert modular variety. Blasius instead considers the symmetric square of the automorphic representation of $\Gl_{2,F}$ associated to the Hilbert eigenform; it is an automorphic representation of $\Gl_{3,F}$ and its base change to a quadratic imaginary field appears in the middle degree cohomology of a Picard modular variety. He then considers the associated motive and shows that its Betti realisation is the symmetric square of a polarised Hodge structure of type $(1,0), (0,1)$. This gives a complex abelian variety $A$ and Conjecture \ref{absolutehodge} (applied to the product of the Picard modular variety and $A$) allows to conclude that $A$ is defined over a number field containing $F$. He then finds the desired abelian variety inside the restriction of scalars of $A$ over $F$. 
\end{rmk}

\begin{rmk}\label{rmk2}
Theorem \ref{generalblasius} is known to hold unconditionally in many interesting cases (for example when $n_F$ is odd, by the work of Hida). For more details we refer to \cite[Theorem 3]{MR2058605} and references therein. The proof of such unconditional cases actually follows Shimura's proof of Theorem \ref{eichlershimurathm}, rather than the strategy described in the previous remark.
\end{rmk}
%??? what if $n_F>1$, check what the ref really use 
\subsection{A remark on polarisations}\label{polar}
To use Theorems \ref{eichlershimurathm} and \ref{generalblasius} to produce $F$-points of a Hilbert modular variety, we need of course the abelian varieties produced to be principally polarised (up to isogeny would actually be enough for our applications, if the isogeny is defined over the base field $F$). An abelian variety over an algebraically closed base field always admits an isogeny to a principally polarised abelian variety. But since the same does not hold over number fields, some considerations are needed. The first observation is that every weight one Hodge structure of dimension $2$ with an action by $\Oo_{K(f)}$ is automatically $\Oo_{K(f)}$-polarised, as explained for example in \cite[Appendix B]{hilbertbas}.

As noticed in \cite[Remark 5.7.]{MR2058605}, Blasius first finds a principally polarised abelian variety $A$ over a finite extension $L'/F$. Actually we can assume that $A$ has a principal $\Oo_{K(f)}$-polarisation $\lambda$. As explained above, the proof considers then the Weil restriction of $A$ to $F$, which is again principally $\Oo_{K(f)}$-polarised. It is not hard to see that the construction of \cite[section 7]{MR2058605} behaves well with respect to the $\Oo_{K(f)}$-polarisation and therefore the proof actually produces an $\Oo_{K(f)}$-polarised abelian variety over $F$.

%\begin{rmk}?? maybe it is enough to say that it is a weight one HS of dimension 2 over $\Oo_{K(f)}$, and so it is automoatically $\Oo_{K(f)}$-polarised and then also the galois reps has always coeff in $\Gl_2$, so the isogeny $A\to A^\vee /F$ has to be an iso. \end{rmk}
\section{Producing abelian varieties via Serre's conjecture}\label{sectionserre}
In this section we prove Theorems \ref{main1} and \ref{main2}. As in Section \ref{mainresults}, consider two totally real fields $L$ and $K$. We work with a compatible system of Galois representations of $G_L$ with values in $\Res_{K/\Q}(\Gl_2) (\mathbb{A}^f_\Q)=\Gl_2(\mathbb{A}_{K}^f)$, where $\mathbb{A}_{K}^f$ denotes the finite adeles of $K$, that ``looks like'' an algebraic point of the Hilbert modular variety for $K$. We then produce a Hilbert modular form for $L$ of weight $(2,\dots,2)$, and of opportune explicit conductor. Eventually we obtain an abelian variety over $L$ that allows us to produce an $L$-rational point on the Hilbert modular variety for $K$. 

\subsection{Weakly compatible systems}
The definition of weakly compatible families presented is due to Serre, who called them \emph{strictly compatible} in \cite[Page I-11]{serrebook}. It follows from the Weil conjectures that the $\ell$-adic Tate modules of abelian varieties form a weakly compatible system of Galois representations.

 \begin{defi}[Weakly compatible system]\label{weakcomp}
A system $\{\rho_v:G_L\to \Gl_{2}(K_v)\}_v$ is \emph{weakly compatible} if there exists a finite set of places $S$ of $L$ such that:
\begin{itemize}
\item[(i)] For all places $w$ of $L$, $\rho_v$ is unramified outside the set $S_v$. Here we denote by $S_v$ the union of $S$ and all the primes of $L$ dividing $\ell$ where $\ell$ is the residue characteristic of $K_v$;
\item[(ii)] For all $w \notin S_v$, denoting by $\frw$ a Frobenius element at $w$, the characteristic polynomial of $\rho_v(\frw)$ has $K$-rational coefficients and it is independent of $v$. 
\end{itemize}
\end{defi}
Recall that $\rho_v$ is said to be unramified at a place $w$ of $L$ if the image of the inertia at $w$ is trivial. If $\rho_v$ is attached to the $v$-adic cohomology of a smooth proper variety defined over a number field, the smooth and proper base change theorems, see for example \cite[I, Theorem 5.3.2 and Theorem 4.1.1]{MR463174}, imply that $\rho_v$ is unramified at every place $w\notin S_v$ such that $X$ has good reduction at $w$.

\subsection{Key proposition}\label{keyprop}
Fix a finite set of places $S$ of $L$, containing all the archimedean places. To prove Theorem \ref{main1} and Theorem \ref{main2} we need the following.
\begin{prop}\label{propmodform}
Assume Conjecture \ref{ourserreconj} and let $\{\rho_v\}_v$ be a system of representations satisfying conditions $(\mathcal{S}.1)$-$(\mathcal{S}.4)$. For every $w\not\in S$, let $a_w\in\Oo_K$ be the trace of $\rho_v(\operatorname{Frob}_w)$. Then there exists $f$ a normalised Hilbert eigenform for $L$ with Fourier coefficients in $\Oo_K$, such that for every $w\not\in S$, $a_w(f)=a_w$. Moreover $f$ is of weight $(2,\dots,2)$ and conductor divisible only by primes in $S$. 
\end{prop}

\begin{rmk}
If we start with an abelian variety $A/L$ with $\Oo_K$-multiplication, we can produce a system 
\begin{displaymath}
\rho_v: G_L \to \Gl(T_v(A))
\end{displaymath}
which satisfies the four conditions of (\ref{conditions}) for $S$ the union of infinite places and the set of places of bad reduction, see \cite[$\S$3]{ribet}. Proposition \ref{propmodform} hence implies that $A$ is modular, i.e. there exists a Hilbert modular form for the totally real field $L$ such that
\begin{displaymath}
L(A/L,s)=L(f,s)
\end{displaymath}
up to a finite number of Euler factors. Unconditionally, it has been proven that elliptic curves over real quadratic fields are modular (see \cite{realquadrmod}) and, more in general, the work of Taylor and Kisin implies that elliptic curves over $L$ become modular (in this sense) after a totally real extension $L'/L$. See \cite[Theorem 1.16]{MR2905534} and reference therein.
\end{rmk}

In the proof of Proposition \ref{propmodform} we use Conjecture \ref{ourserreconj}, and the following result due to Serre (for the proof see \cite[4.9.4]{serresur}).
\begin{prop}[Serre]\label{conductorprop}
Let $q$ be a power of $\ell$. Let $r: G_E \to \Gl_2(\Ff_q)$ a continuous homomorphism, where $q=\ell^t$ and $E$ is a local field of residue characteristic $p\neq \ell$ and discrete valuation $v_E$. Let $e_E:=v_E(p)$ and $c\geq 0$ be an integer such that the image via $r$ of the wild inertia of $E$ has cardinality $p^c$. We denote by $n(r,E)$ the exponent of the conductor of $r$. We have
\begin{displaymath}
n(r,E)\leq 2\left(1+e_E\cdot c + \tfrac{e_E}{p-1}\right).
\end{displaymath}
\end{prop}
We also need to compute the weight and the conductor of the modular forms produced by Conjecture \ref{ourserreconj}. As anticipated in Remark \ref{weak}, this is a known result under some assumptions. The weight part stated in the following theorem is a special case of the work \cite{geeliusavitt}; the conductor part follows from automorphy lifting methods or can be seen as a consequence of the main theorem of \cite{fuji}.

\begin{thm}[\cite{geeliusavitt,fuji}]\label{thmweightcond} Let $\ell>5$ and $\bar{\rho}: G_L\to \Gl_2(\overline{\mathbb{F}}_{\ell})$ be an irreducible totally odd representation such that its determinant is the cyclotomic character and it is finite flat at all places $w\mid\ell$. Assume furthermore that $\bar{\rho}$ satisfies the Taylor--Wiles assumption, namely
 \begin{equation}\label{taylorwiles}
\tag{TW}
\bar{\rho}_{\mid G_{L(\zeta_{\ell})}} \text{ is irreducible},
\end{equation}
where $\zeta_{\ell}$ is a primitive $\ell$th root of unity. Then if $\bar{\rho}$ is modular, there exists a Hilbert modular form of parallel weight 2 and conductor equal to the Artin conductor of $\bar{\rho}$ giving rise to $\bar{\rho}$.
\end{thm}

We now ready to prove Proposition \ref{propmodform}.
\begin{proof}[Proof of Proposition \ref{propmodform}]
Our goal is to apply Serre's conjecture to $\bar{\rho}_v$, the reduction modulo $v$ of the representation $\rho_v$, for infinitely many $v\notin S_K$, where $S_K$ is the following finite set of primes of $K$ 
\begin{displaymath}
S_K=\{ v: \ v\mid \ell \text{ and } w\mid \ell \text{ for some }w\in S \text{ or $\ell$ is ramified in $L$}\}.
\end{displaymath} 
Let $v$ be such a prime, let $\ell$ be its residue characteristic. 

We first compute the conductor $N(\bar{\rho}_v)$. Since $\bar{\rho}_v$ is unramified outside $S_v$, the conductor is divisible only by primes in $S$. We then apply Proposition \ref{conductorprop} to $E=L_w$ and $r=\bar{\rho}_v$. The image of $\bar{\rho}_v$ s contained in $\Gl_2(\Ff_{\ell^t})$, where $t\leq [K:\Q]=n_K$. The cardinality of this group is $\ell^t(\ell^{2t}-1)(\ell^t-1)$. Let $W_w$ denote the wild inertia subgroup of $G_{L_w}$. If $\ell$ satisfies the following congruences 
\begin{equation}\label{congruences}
\tag{$\star$}
\ell^{n_K} \not\equiv \begin{cases}
\pm 1 \mod p &\text{ if } p\neq 2,3\\
\pm 1 \mod 8 &\text{ if } p=2\\
\pm 1, 4,7 \mod 9 &\text{ if } p=3,
\end{cases}
\end{equation}
then the same congruences hold for $\ell^t$ and hence $\bar{\rho}_v(W_w)$ is trivial if $p\neq 2,3$ and is at most $p^{5}$ if $p=2$ and at most $p$ if $p=3$. Hence for $v\notin S$ laying above $\ell$ satisfying the above conditions, using that $e_E\leq [L:\Q]=n_L$, the inequality of Proposition \ref{conductorprop} implies
\begin{displaymath}
n_K(\bar{\rho}_v,L_w)\leq \begin{cases}
2(1+n_L) &\text{ if } p\neq 2,3\\
2(1+6n_L) &\text{ if } p=2\\
2(1+2n_L) &\text{ if } p=3.
\end{cases}
\end{displaymath}
Writing, $\mathfrak{p}_w$ for the prime ideal of $L$ corresponding to $w$, we hence find that the conductor of $\bar{\rho}_v$ divides 
\begin{displaymath}
\mathfrak{C}:= \prod_{\substack{w\in S,\\ w\nmid 2,3}} \mathfrak{p}_w^{2+2n_L}\cdot \prod_{\substack{w\in S,\\ w\mid 2}}\mathfrak{p}_w^{2+12n_L}\cdot \prod_{\substack{w\in S,\\ w\mid 3}}\mathfrak{p}_w^{2+4n_L}.
\end{displaymath}
Finally, notice that $\rho_v$ is odd thanks to the condition on the determinant and, moreover, \cite[Proposition 5.3.2]{MR3152941} implies that there exists a density one set of primes such that $(\bar{\rho}_v)_{\mid G_{L(\zeta_{\ell})}}$ is irreducible, i.e. (\ref{taylorwiles}) is satisfied.

We can apply Conjecture \ref{ourserreconj} to $\bar{\rho}_v$ for $v\in \Sigma$, where $\Sigma$ is the infinite set of primes $v\mid \ell$ such that $v\not\in S_K$, $\ell$ satisfies (\ref{congruences}), $\bar{\rho}_v$ is absolutely irreducible and satisfies (\ref{taylorwiles}). We have produced infinitely many $f_v$ Hilbert modular eigenform, which by Theorem \ref{thmweightcond} are of parallel weight 2 and level dividing $\mathfrak{C}$. Their Fourier coefficients are defined over a ring $\Oo(v)\subset \Oo_K$ and at a prime $\lambda\mid v$ the associated Galois representation $\rho_{f_v,\lambda}$ is isomorphic to $\bar{\rho}_v$ modulo $\lambda$. Since the space of Hilbert modular form of fixed weight and with conductor dividing $\mathfrak{C}$ is finite dimensional (see \cite[Theorem 6.1]{freitag}), we can find at least one Hilbert modular eigenform $f$ of parallel weight 2 and level dividing $\mathfrak{C}$ defined over some $\Oo\subset\Oo_K$ such that for infinitely many of the $v$ above the same property holds for $\rho_{f,\lambda}$ for $\lambda\mid v$. This implies that for all $w\not\in S$ the congruence
\begin{displaymath}
a_w(f)\equiv a_w \mod \lambda\mid v
\end{displaymath}
holds for infinitely many primes $\lambda$ and hence $a_w(f)= a_w$, as required. 
\end{proof}

\subsection{Proof of Theorem \ref{main}}
Recall that, as in Section \ref{mainresults}, the system $(\rho_v)_v$ is required to satisfy the following additional property:
\begin{equation}\label{allK}
\tag{$\mathcal{S}.5$} 
\text{the field generated by $a_w$ for every $w$ is exactly $K$}. 
\end{equation}
In other words, we have $K(f)=K$, where $f$ is the Hilbert modular form for $L$ produced in Proposition \ref{propmodform}.

\begin{proof}[Proof of Theorem \ref{main}] Starting with our initial datum of Galois representations, we have produced a Hilbert modular form $f$ for $L$. We can then apply Theorem \ref{generalblasius}, which gives an abelian variety $A_f$ over $L$ of dimension $[K:\Q]$ and an embedding of $\Oo_K$ into $\End(A)$. For all but finitely many $w\mid p$ at which $A_f$ has good reduction 
\begin{displaymath}
\det(1-X\rho_{A_f,v}(\fr_w))=1-a_w(f)X+N_wX^2,
\end{displaymath}
where $v$ is a finite prime of $K$ not dividing $p$ and $\rho_{A_f,v}$ is the $G_{L}$-representation on $T_v(A_f)$, the $v$-adic Tate module of $A_f$. We therefore have produced an abelian variety $A_f$ as stated in Theorem \ref{main1} and \ref{main2}.

We just need to stress that we don't require any conjectural statement in the case $n_L=1$. Indeed we can use Theorem \ref{eichlershimurathm} in place of Blasius' conjectural version, and Serre's conjecture is fully known thanks to the work of Khare--Wintenberger \cite[Theorem 1.2]{MR2480604}.
\end{proof}

\subsection{A corollary}
We rephrase the main results of the section as needed to prove Theorem \ref{maincor}.
\begin{cor}\label{corfordescent}
Assume that, in the setting of Theorems \ref{main1} and \ref{main2}, we also have a representation
\begin{displaymath}
\rho : G_L\to \Gl_2(\Oo_K/\mathfrak{N}),
\end{displaymath}
for some integral ideal $\mathfrak{N}\subset \Oo_K$, such that for all pairs $(v, a)$, where $v$ is a place of $K$ and $a\in \N-\{0\}$, such that $v^a$ divides $\mathfrak{N}$, the reductions of $\rho$ and $\rho_v$ modulo $v^a$ agree. Then there exists a $n_K$-dimensional abelian variety $A/L$ with good reduction at all $w$ outside $S$ and the action of $G_L$ on $A[\mathfrak{N}]$ is given by $\rho$.
\end{cor}
\begin{proof}
Theorems \ref{main1} and \ref{main2} give an $n_K$-dimensional abelian variety $A/L$ and, using the Néron-Ogg-Shafarevich criterion, we can see that it has good reduction at all $w$ outside $S$. Finally $G_L$ acts on $A[\mathfrak{N}]$ via $\rho$, since the reduction modulo $v^a$ of $\rho$ and $\rho_v$ agree.
\end{proof}

\section{Finite descent obstruction and proof of Theorem \ref{maincor}}\label{finitesection}
In this final section we recall the finite descent obstruction for integral points, explaining how it relates to the system of Galois representations considered in the previous section. Using Theorem \ref{main}, we indeed produce an $\Oo_{L,S}$-point of integral models of twists of Hilbert modular varieties, which is what we require to prove Theorem \ref{maincor}.
\subsection{Recap on the integral finite descent obstruction}\label{recapobs}
Let $Y/F$ be a smooth, geometrically connected variety (not necessarily proper) over a number field $F$. Let $S$ be a finite set of places of $F$ and, as before, assume that $S$ contains the archimedean places and all places of bad reduction of $Y$. Choose and fix a smooth model $\mathcal{Y}$ of $Y$ over $\Oo_{F,S}$. In this section we recall the definition of the set $\mathcal{Y}^{f-cov}$, which will correspond to the adelic points of $\mathcal{Y}$ that are unobstructed by all Galois covers. To make the paper self contained we recall the discussion from \cite[Section 2]{helmvoloch} (where the authors work with affine curves). We then recall that, for Hilbert modular varieties, a point unobstructed by finite covers admits an infinite tower of twists of covers with a compatible system of lifts of adelic points along the tower (following \cite[Proposition 1]{helmvoloch}).

Let $\pi: \mathcal{X} \to \mathcal{Y}$ be a map of $\Oo_{F,S}$-schemes, such that it becomes a Galois covering over $\overline{F}$. Such map is called geometrical Galois cover of $\mathcal{Y}$. Denote by $\Tw(\pi)$ the set of isomorphism classes of twists of $\pi$, i.e. of maps $\pi':\mathcal{X}'\to \mathcal{Y}$ that become isomorphic to $\pi$ over $\overline{F}$. We have
\begin{displaymath}
\mathcal{Y}(\Oo_{F,S})= \bigcup_{\pi ' \in \Tw_0(\pi)} \pi' \left(\mathcal{X}'(\Oo_{F,S})\right),
\end{displaymath}
where $\Tw_0(\pi)$ is a suitable finite subset of $\Tw(\pi)$ (for a more detailed discussion we refer to \cite[Page 105 and 106]{MR1845760}), and $\pi':\mathcal{X}'\to \mathcal{Y}$ is a twist of $\pi$. In what follows, $w$ denotes a place of $F$.

\begin{defi}
We define $\mathcal{Y}^{f-cov}(\Oo_{F,S})=\mathcal{Y}^{f-cov}$ as the set of $ (P_w)_w \in \prod _{w \notin S} \mathcal{Y}(\Oo_{F_w})$ such that, for each geometrical Galois cover $\pi$, we can write
\begin{displaymath}
P_w= \pi' (Q_w), \ \ \ \forall w \notin S
\end{displaymath}
for some  $\pi ' \in \Tw_0(\pi)$ and $(Q_w)_w\in \prod _{w \notin S} \mathcal{X}'(\Oo_{F_w})$.
\end{defi}
\begin{prop}\label{equivalence}
A point $(P_w)_w $ lies in $\mathcal{Y}^{f-cov}$ if and only if, for each geometrical Galois cover $\pi: \mathcal{X} \to \mathcal{Y} $, we can choose a twist $\pi': \mathcal{X}'\to \mathcal{Y}$ and a point $(P_w)_\pi \in \prod_{w\notin S}\mathcal{X}'(\Oo_{F_w})$ lifting $(P_w)_w$ in a compatible way (i.e. if $\pi_1 , \pi_2$ are Galois covers and $\pi_2$ dominates $\pi_1$, then $\pi_2'$ dominates $\pi_1'$ and $(P_w)_{\pi_2'}$ maps to $(P_w)_{\pi_1'}$). 
\end{prop}
A few words to justify the equivalence between the two definitions are needed. This is explained in \cite[Proposition 1]{helmvoloch} for curves, and it relies on results from \cite{MR2368954} (notably \cite[Lemma 5.7]{MR2368954}\footnote{It is actually better to refer to the corrected version of \cite{MR2368954} available on the author's website (\url{www.mathe2.uni-bayreuth.de/stoll/papers/Errata-FiniteDescent-ANT.pdf}).}). In \cite{MR2368954}, Stoll works with projective varieties and their rational points, but what he says still holds true for the integral points of non-projective varieties. Once such differences are taken into account, the proof works in the same way in our setting.

%\begin{proof}[Sketch of the proof of Proposition \ref{equivalence}]
%The proof now goes on as \cite[Proposition 1]{helmvoloch} (where $X$ is our $Y$ and $Y$ is our $X$).
%\end{proof}

Clearly we have $\mathcal{Y}(\Oo_{F,S})\subset \mathcal{Y}^{f-cov}$ and so if $ \mathcal{Y}^{f-cov}$ is empty, then $\mathcal{Y}(\Oo_{F,S})$ has to be empty as well. What can be said when $ \mathcal{Y}^{f-cov}$ contains a point?
\begin{defi} If $\mathcal{Y}^{f-cov}(\Oo_{F,S})\neq \emptyset$ implies that $\mathcal{Y}(\Oo_{F,S})$ is non-empty, we say that the $S$-integral finite descent obstruction is the only obstruction for the existence of $S$-integral points.
\end{defi}
From now on we specialise to the case of Hilbert modular varieties (and their twists).
\subsection{Integral points on Hilbert modular varieties}\label{integralmodels}
Recall the notation from section \ref{modulispace}. Let $Y_K(\mathfrak{N})$ be the $n_K$-dimensional $\Q$-scheme described in Section \ref{modulispace} and let $N$ be the integer such that $\mathfrak{N}\cap \Z=(N)$. The set of twists of $\pi:  Y_K(\mathfrak{N}) \to Y_K(1)$ over a number field $F$ corresponds to the set of Galois representations $\rho : G_F \to \Gl_2( {\Oo_K}/\mathfrak{N})$ whose determinant is the cyclotomic character $\chi : G_F \to (\Z/N\Z)^{\times}$. Moreover a point $x \in Y_K(1)(F)$ lifts to a $F$-rational point of the twist of $Y_K(\mathfrak{N})$ corresponding to a representation $\rho$, if and only if $\rho$ describes the action of $G_F$ on the $\mathfrak{N}$-torsion of the underlying abelian variety $A_x$ (as an $\Oo_K$-module).

Using the moduli interpretation of Section \ref{modulispace}, we can construct a model of $Y_K(\mathfrak{N})$ over $\Z$, which is smooth over $\Z[1/b]$, for some natural number $b$, divisible by $N$. To be more precise $b$ depends on the level structure and the discriminant of $K$. Fixing such model, that we denote by $\mathcal{Y}_K(\mathfrak{N})$, we can talk about $\Oo_{F,S}$-points of $Y_K(\mathfrak{N})$, for any number field $F$ and set of places $S$ containing the archimedean places and the ones dividing $b$. Such $\Oo_{F,S}$-points then correspond to abelian varieties (with some extra structure), having good reduction outside $S$. Recall that $N$ is assumed to be bigger than $3$, since it is important to have a fine moduli space. For example the affine line is the moduli space of elliptic curves and has plenty of $\Z$-points, even though there are no elliptic curves defined over $\Z$.

We are ready to study the finite descent obstruction for the $\Oo_{L,S}$-points of $\mathcal{Y}_K(\mathfrak{N})$ and its twists, where $L$ is a totally real field (the fact that $L$ is totally real will be used only in the next section). Let 
\begin{displaymath}
 \rho : G_L\to \Gl_2(\Oo_K/\mathfrak{N}).
\end{displaymath}
be a representation whose determinant is the cyclotomic character. Assume that $S$ contains the places $w$ of ramification of $\rho$. Under this assumption, arguing as above, we can consider $\mathcal{Y}_\rho$ the $S$-integral model of the twist of $\mathcal{Y}_K(\mathfrak{N})$ corresponding to $\rho$. From now on we assume that $\mathcal{Y}_\rho(\Oo_{L,S})$ is non-empty. The next lemma relates a point $(P_w)_w \in\mathcal{Y}_\rho^{f-cov}$ to a system of Galois representations as considered in the previous section.

\begin{lemma}\label{corex2}
A point $(P_w)_w \in\mathcal{Y}_\rho^{f-cov}(\Oo_{L,S})$ corresponds to the following data:
\begin{itemize}
\item For each finite place $v$ of $K$ a representation $\rho_v : G_L \to \Gl_2(K_v)$;
\item For each finite place $w$ of $L$ such that $w\notin S$ an abelian variety $A_w/ L_w$ of dimension $n_K$, with good reduction and $\Oo_K$-multiplication;
\end{itemize}
satisfying:
\begin{itemize}
\item For every place $v$ in $K$ the action of $G_{L_w}$ on $T_v(A_w)$ is given by the restriction of $\rho_v$ to the decomposition group at $w$;
\item For all pairs $(v, a)$ such that $v^a$ divides $\mathfrak{N}$, the reductions of $\rho$ and $\rho_v$ modulo $v^a$ agree.
\end{itemize}
Moreover every such system satisfies the first four conditions of (\ref{conditions}).
\end{lemma}
\begin{proof}
We first check, using Lemma \ref{equivalence}, that an unobstructed point corresponds to a system of Galois representations as described above, and then we show that every such system enjoys the desired properties. 

Thanks to Proposition \ref{equivalence}, we can fix a compatible system of lifts of $(P_w)_w$ on $\mathcal{Y}_\rho^{f-cov}(\Oo_{L,S})$. In particular, for each $\mathfrak{M}$ divisible by $\mathfrak{N}$, we obtain a twist $\mathcal{Y}_K(\mathfrak{M})'$ of $\mathcal{Y}_K(\mathfrak{M})$ and a compatible family of points $(P_w)_\mathfrak{M}$ of $\mathcal{Y}_K(\mathfrak{M})'$ lifting $(P_w)_w$. We remark here that the latter compatible family of points depends on $(P_w)_w$. Indeed, a priori, we can not simply lift $\rho$ to mod $\mathfrak{M}$ coefficients.

By the interpretation of $\mathcal{Y}_K (\mathfrak{M})'$ as moduli space of abelian varieties discussed above, the point $(P_w)_\mathfrak{M}$ corresponds to an abelian variety $A_w/ L_w$ of dimension $n_K$, with good reduction and $\Oo_K$-multiplication and prescribed $\mathfrak{M}$-torsion. The other conditions are easily checked as at the end of proof of \cite[Theorem 2]{helmvoloch}.

The fact that the action of $G_{L_w}$ on $T_v(A_w)$ is given by the restriction of $\rho_v$ to the decomposition group at $w$ ensures that ($\mathcal{S}.1$) and ($\mathcal{S}.2$) are satisfied. Moreover, since $A_w$ has good reduction, then $A_w[v]\simeq \bar{\rho}_v$ is a finite flat group scheme over $O_{K_w}$ for all $w\mid \ell$ if $v\mid \ell$; this implies that ($\mathcal{S}.3$) also holds. 

Finally, we need to show that ($\mathcal{S}.4$) is satisfied. With the three conditions above one can show, as in the proof of Proposition \ref{propmodform}, that the conductor of $\bar{\rho}_v$ divides a fixed ideal $\mathfrak{C}$ of $L$. If $w$ is such that $A_w/L_w$ is supersingular at $v$, then $A_w[v]\simeq \bar{\rho}_v$ is absolutely irreducible. If there existed infinitely many $v$ such that $\bar{\rho}_v$ is absolutely reducible, we could then write the semisemplification of $\bar{\rho}_v$ as direct sum of $\phi$ and $\chi_{\ell}\phi^{-1}$, for some character $\phi$. Since $\bar{\rho}_v\simeq A_w[v]$ for $w\mid p$, we then have that $A_w$ is ordinary and hence $\phi$ is unramified at $w$. We also know that the conductor of $\phi$ divides $\mathfrak{C}$, hence if $w\equiv 1 \mod \mathfrak{C}$ we have 
\begin{displaymath}
a_w(A_w):=\Tr(\operatorname{Frob}_w, T_v(A_w)) \equiv \chi_{\ell}(\operatorname{Frob}_w)+1 \mod v.
\end{displaymath}
Since $\chi_{\ell}(\operatorname{Frob}_w)=p^{[L_w:\Q_p]}$, we showed that if we had infinitely many $v$ such that $\bar{\rho}_v$ is absolutely reducible, we would find $a_w(A_w)=p^{[L_w:\Q_p]}+1$. Since the Weil bound says that $|a_w|\leq 2\sqrt{p^{[L_w:\Q_p]}}$, we reached a contradiction. 
\end{proof}

\begin{rmk}\label{remkfalting}
As discussed above, for any number field $F$, we have a map from $Y_K(F)$ to systems of Galois representations satisfying ($\mathcal{S}.1)-(\mathcal{S}.4$). Thanks to Faltings \cite[Satz 6]{MR718935}, this map has finite fibres. Indeed if two points give rise to the same system, the two corresponding abelian varieties have the same locus of bad reduction, that we denote by $S$. It follows from the Shafarevich conjecture that Shimura varieties of abelian type have only finitely $\Oo_{F,S}$-points. For more details we refer to \cite[Theorem 3.2(A)]{MR2109988}.
\end{rmk}
We are now ready to prove the main theorem about descent obstruction for Hilbert modular varieties. 
\subsection{Proof of Theorem \ref{maincor}}\label{proofmaincor}
We do not treat the cases $n_L=1$ and $n_L>1$ separately, but, as in the proof of Theorem \ref{main}, we emphasize that we don't need any conjectural statement in the case $n_L=1$, since we have Shimura's unconditional result, Theorem \ref{eichlershimurathm}.

\begin{proof}[Proof of Theorem \ref{maincor}]
Thanks to Lemma \ref{corex2}, a point in $\mathcal{Y}_\rho^{f-cov} (\Oo_{L,S})$, which is assumed to be non-empty, gives rise to a compatible system of representations of $G_L$, denoted by $\{\rho_v\}_v$. Let $E$ be the the subfield of $K$  generated by $\tr (\rho_v(\operatorname{Frob}_w))$ for all $w$. If $E=K$, Corollary \ref{corfordescent} produces an $\Oo_{L,S}$-abelian variety $A$ with $\Oo_K$-multiplication, such that $G_L$ acts on $A[\mathfrak{N}]$ via $\rho$. To conclude the proof we just need to see how $A$ corresponds to a point $P\in \mathcal{Y}_{\rho}(\Oo_{L,S})$. The only issue that is not clear from the quoted corollary is whether $A$ is principally $\Oo_K$-polarised, but this follows from the discussion in section \ref{polar}.

If $E$ is strictly contained in $K$, i.e. condition (\ref{allK}) is not satisfied, we consider $S_E$, the Hilbert modular variety associated to $E$ and of level $\mathfrak{N} \cap \Oo_{E}$. For the same reason as above, the system $\{\rho_v\}$ corresponds to a $\Oo_{L,S}$-point $P$ of the twist by $\rho$ of $S_E$. The embedding $\Res_{E/\Q}\Gl_2\hookrightarrow \Res_{K/\Q}\Gl_2$ induces a map of Shimura varieties
\begin{displaymath}
r :S_E \to Y_K(\mathfrak{N}),
\end{displaymath}
and therefore on their twists by $\rho$. Via $r$, we can regard $P$ as an $\Oo_{L,S}$-point of $\mathcal{Y}_\rho$, which completes the proof of the theorem. The only difference is that the abelian variety constructed in this case is not primitive. This concludes the proof of Theorem \ref{maincor}.
\end{proof}

%\begin{rmk}
%Is such $P$ the unique inducing $P_w$???? there are at most finitely many possibilities, check if the \emph{adelic} Tate conj, gives uniqueness. Usually we need trivial endomorphism, but here everything here is $\Oo_K$-linear.. To have some hope to do this, we should start with reps in $\Oo_v$, not $K_v$
%\end{rmk}

\bibliographystyle{abbrv}
\bibliography{biblio}

\Addresses

\end{document}